\documentclass[11pt]{amsart}

\usepackage{latexsym}
\usepackage{amssymb}
\usepackage{amsmath}
\usepackage{enumerate}

\newtheorem{theorem}{Theorem}[section]
\newtheorem{lemma}[theorem]{Lemma}
\newtheorem{proposition}[theorem]{Proposition}

\newtheorem{remark}[theorem]{Remark}

\newcommand\Alg{\mathop{\rm Alg}}
\newcommand\Lat{\mathop{\rm Lat}}

\newcommand\ball{\mathop{\rm Ball}}

\newcommand\Conv  {\mathop{\rm Conv}}

\newcommand{\cl}[1]{\mathcal{#1}}
\newcommand{\bb}[1]{\mathbb{#1}}

\begin{document}
\title{Norms of vector functionals}
\date{31 May 2018}

\author[M. Anoussis]{M. Anoussis}
\address{Department of Mathematics,
  University of the Aegean, 83200 Samos, Greece}
\email{mano@aegean.gr}

\author[N. Ozawa]{N. Ozawa}
\address{
Research Institute for Mathematical Sciences, Kyoto University, Kyoto 606-8502, Japan}
\email{narutaka@kurims.kyoto-u.ac.jp}

\author[I. Todorov]{I. G. Todorov}
\address{Mathematical Sciences Research Centre,
  Queen's University Belfast, Belfast BT7 1NN, United Kingdom, and 
  School of Mathematical Sciences, Nankai University,
  94 Weijin Road, Tianjin, 300071, P.R. China}
\email{i.todorov@qub.ac.uk}


\begin{abstract}
We examine the question of when, and how, the norm of a vector functional on an operator algebra
can be controlled by the invariant subspace lattice of the algebra. We introduce a related
operator algebraic property, and show that it is satisfied 
by all von Neumann algebras and by all CSL algebras.
We exhibit examples of operator algebras that do not satisfy the 
property or any scaled version of it. 
\end{abstract}

\maketitle

\section{Introduction and preliminaries}\label{s_intro}

Vector functionals have played a fundamental role in the theory of operator algebras 
since its inception. In the area of selfadjoint algebras, they arise naturally as 
pure states and through the GNS construction \cite{kr}. In the realm of non-selfadjoint operator algebras, 
they are at the base of the notions of reflexivity \cite{ls} and hyperreflexivity \cite{a10}. 
In this note, we study the question of whether the norm of a vector functional on a given operator 
algebra can be controlled through the invariant subspace lattice of the algebra. 
We propose a new bound for the norms of vector functionals
on certain classes of operator algebras,  establishing
minimax inequalities that capture the duality between the algebras and their subspace lattices. 

In order to describe our results in more detail, we introduce some notation. 
Let $\cl B(H)$ be the collection of all bounded linear operators acting on a Hilbert space $H$ and 
$\cl P(H)$ be the set of all projections, that is, self-adjoint idempotents, on $H$.
If $\cl A\subseteq \cl B(H)$ is a unital operator algebra,
let 
\begin{equation}\label{eq_lat}
\Lat\cl A = \{L\in \cl B(H) : \mbox{ projection}, (I-L)\cl A L = \{0\}\}
\end{equation}
be the invariant subspace lattice of $\cl A$. 
Given vectors $x,y\in H$, let $\omega_{x,y}$ be the vector functional on $\cl B(H)$ defined by
$\omega_{x,y}(T) = (Tx,y)$, and let $\omega_{x,y}|_{\cl A}$ be the restriction of $\omega_{x,y}$ to $\cl A$.

We are interested in the question for which operator algebras $\cl A$ the inequality
\begin{equation}\label{eq_el}
\inf\mbox{}_{L\in \Lat\cl A} \left(\|(I-L)x\|^2 + \|Ly\|^2\right) \leq \|\omega_{x,y}|_{\cl A}\|
\end{equation}
holds for all vectors $x,y\in H$. 
If (\ref{eq_el}) is satisfied, we say that $\cl A$ possesses \emph{property (V)}. 
We prove that all von Neumann algebras and all CSL algebras
possess property (V). 
We introduce a scaled version of the property, in which (\ref{eq_el}) holds up to 
a constant, and show that the algebra of operators leaving two non-trivial 
closed subspaces in generic position invariant satisfies it if and only if the angle between the subspaces is positive. 
In particular, this implies that if the angle between the subspaces is zero then 
the corresponding algebra does not possess (V).

\smallskip

In the rest of this section, we fix notation and recall some notions needed in the sequel. 
We fix throughout a Hilbert space $H$. 
For a subset $\cl S\subseteq \cl B(H)$, we write ${\rm Ball}(\cl S)$ 
for the set of all contractions in $\cl S$. 
We recall the weak operator topology on $\cl B(H)$, denoted here by $w$, 
in which a net $(A_i)_i$ converges to an operator $A$ if and only if 
$\omega_{x,y}(A_i)\to \omega_{x,y}(A)$ for all vectors $x,y\in H$, and
the strong operator topology, denoted by $s$, in which 
$(A_i)_i$ converges to an operator $A$ if and only if 
$A_i x\to Ax$ for every vector $x\in H$.
We note that ${\rm Ball}(\cl B(H))$ is $w$-compact.

If $L$ is a projection, we let 
as usual $L^{\perp} = I - L$.
The natural order in $\cl P(H)$ is the order of a (complete) lattice; by a 
\emph{subspace lattice} we will mean a sublattice of $\cl P(H)$ that is 
closed in the strong operator topology. 
It is easily verified that, if $\cl A\subseteq \cl B(H)$ is a
unital (and not necessarily selfadjoint) operator algebra then the set 
$\Lat\cl A$ defined in (\ref{eq_lat}) is a subspace lattice. 
A \emph{commutative subspace lattice (CSL)} is a subspace lattice 
$\cl L$ such that $PQ = QP$ for all $P,Q\in \cl L$. A \emph{CSL algebra} is an  
algebra of the form $\Alg\cl L$ for some CSL $\cl L$, where 
$$\Alg\cl L = \{A\in \cl B(H) : L^{\perp}A L = 0, \mbox{ for all } L\in \cl L\}$$
is the algebra of all operators in $\cl B(H)$ leaving the ranges of projections in $\cl L$ invariant.
We refer the reader to \cite{a} and \cite{dav-book} for a background on CSL's and CSL algebras.

We finish this section with a reformulation of (\ref{eq_el}). Given a subspace lattice $\cl L$ on $H$, let
$$\cl E_{\cl L} = \{(x,y) \in H\times H : \mbox{ there exists } \ L\in \cl L \mbox{ with } Lx = x, Ly = 0\}.$$
The left hand side of (\ref{eq_el}) is equal to
$d((x,y),\cl E_{\cl L})^2$, where $d$ denotes the distance in the Hilbert space $H\oplus H$.
Thus, inequality (\ref{eq_el}) becomes
$$d\left((x,y),\cl E_{\cl L}\right)^2\leq \|\omega_{x,y}|_{\cl A}\|.$$

\section{Validity of property (V)}\label{s_v}

If $\cl L$ is a subspace lattice acting on a Hilbert space $H$, we denote by
$\Conv \cl L$ the $w$-closure of the convex hull of $\cl L$; note that 
$\Conv \cl L$ is a ($w$-closed) convex subset of $\cl B(H)$.

\begin{lemma}\label{l_conhulls}
Let $H$ be a Hilbert space, $x,y \in H$, $\cl A\subseteq \cl B(H)$ be a unital operator algebra
and $\cl L = \Lat\cl A$. Then inequality (\ref{eq_el}) is equivalent to 
\begin{equation}\label{eq_elch}
\inf\mbox{}_{A\in \Conv \cl L} \left(((I-A)x,x) + (Ay,y)\right) \leq \|\omega_{x,y}|_{\cl A}\|.
\end{equation}
\end{lemma}
\begin{proof}
Consider the function $f : \Conv\cl L\longrightarrow \bb{R}^+$
given by 
$$f(X) = ((I-X)x,x) + (Xy,y), \ \ \ X\in \Conv\cl L;$$ 
clearly, $f$ is continuous in the weak operator topology.
Moreover, $f$ is affine in the sense that if $X,Y\in\Conv\cl L$ and
$s$ and $t$ are non-negative numbers with $s + t = 1$, then
$f(s X + t Y) = s f(X) + t f(Y)$.
By Bauer's Maximum Principle 
(see \cite[7.69]{ab}), there exists an extreme point
$B$ of $\Conv \cl L$ such that $\min_{X\in \Conv \cl L} f(X) = f(B)$.
By the converse of the Krein-Milman Theorem, $B$ belongs to the weak closure $\overline{\cl L}^{w}$
of $\cl L$. Thus, there exists a net $(P_{\nu})_{\nu}\subseteq  \cl L$ such that
$P_{\nu}\rightarrow_{\nu} B$ weakly; hence
$f(P_{\nu})\rightarrow_{\nu} f(B)$.
It follows that
$$\inf\mbox{}_{L\in \Lat\cl A} \left(\|(I-L)x\|^2 + \|Ly\|^2\right)
\leq \lim\mbox{}_{\nu} \|P_{\nu}^{\perp} x\|^2 + \|P_\nu y\|^2  = f(B),$$
that is, 
$$\inf\mbox{}_{A\in \Conv \cl L} \left(((I-A)x,x) + (Ay,y)\right) \geq \inf\mbox{}_{L\in \Lat\cl A} 
\left(\|(I-L)x\|^2 + \|Ly\|^2\right).$$
Since the reverse inequality is trivial, the claim follows. 
\end{proof}

\noindent {\bf Remark. }
It follows from the proof of Lemma \ref{l_conhulls} that the infimum 
on the left hand side of (\ref{eq_elch}) is attained. 

\medskip

Let $\cl A$ be a von Neumann algebra acting on a Hilbert space $H$, and denote by $\frak{P}$ the 
set of all positive contractions in $\cl A'$. 
By Lemma \ref{l_conhulls}, property (V) is in this case equivalent to the validity of the inequality
\begin{equation}\label{eq_vna}
\inf\mbox{}_{B\in \frak{P}} \left((Bx,x) + ((I-B)y,y)\right) \leq \|\omega_{x,y}|_{\cl A}\|, 
\end{equation}
for all $x,y\in H$.

\begin{lemma}\label{l_reduc}
Let $\cl A\subseteq \cl B(H)$ be a von Neumann algebra.
Let $(x,y)$ and $(\xi,\eta)$ be pairs of vectors in $H$,
and $V\in \cl A'$ be a partial isometry such that $Vx = \xi$, $V^*Vx = x$ and $Vy = \eta$. 
If (\ref{eq_vna}) holds for the pair $(\xi,\eta)$ in the place of $(x,y)$, 
then it also holds for the pair $(x,y)$.
\end{lemma}

\begin{proof}
Set $P = V^*V$. By assumption, there exists $B\in \frak{P}$ such that 
$$\|\omega_{\xi,\eta}|_{\cl A}\| \geq (B\xi,\xi) + ((I-B)\eta,\eta).$$
Let $\tilde{B} = V^*BV + P^{\perp}$; then $\tilde{B}\in \frak{P}$. 
Moreover,
\begin{eqnarray*}
\|\omega_{x,y}|_{\cl A}\| 
& = & 
\sup \left\{(Ax,y) : A\in {\rm Ball}(\cl A)\right\}\\
& = & \sup \left\{(AV^*Vx,y) : A\in {\rm Ball}(\cl A)\right\}\\
& = & 
\sup \left\{(V^*AVx,y) : A\in {\rm Ball}(\cl A)\right\}
= \|\omega_{\xi,\eta}|_{\cl A}\| \\ 
& \geq & (B\xi,\xi) + ((I-B)\eta,\eta)
 = 
(BVx,Vx) + ((I-B)Vy,Vy)\\
& = & 
(V^*BVx,x) + ((P - V^*BV)y,y)\\
& = & 
(\tilde{B}x,x) + ((I - \tilde{B})y,y).
\end{eqnarray*}
\end{proof}

\begin{lemma}\label{l_induc}
Let $\cl A\subseteq \cl B(H)$ be a von Neumann algebra with property (V) and $E$ be a projection in $\cl A'$. 
Then $\cl A|_{EH}$ possesses property (V). 
\end{lemma}

\begin{proof}
Let $x,y \in EH$.
By the assumption, Lemma \ref{l_conhulls} and the subsequent Remark, 
there exists $A\in \frak{P}$ such that
$$(Ax,x) + ((I-A)y,y) \le \left\|\omega_{x,y}|_{\cl A}\right\|.$$
Similarly to the proof of Lemma \ref{l_reduc}, set $B = E^{\perp} + EAE$; then $B\in \frak{P}$ and 
\begin{eqnarray*}
\left\|\omega_{x,y}|_{\cl A E}\right\| & = & 
\left\|\omega_{x,y}|_{\cl A}\right\| \geq (Ax,x) + ((I-A)y,y)\\
& = & 
(EAEx,x) + ((E - EAE)y,y)\\
& = & 
(Bx,x) + ((I-B)y,y).
\end{eqnarray*}
\end{proof}

We will make  use of the following standard fact. 

\begin{lemma}\label{l_pi}
Let $\cl N\subseteq \cl B(H)$ be a von Neumann algebra and $\xi,\eta\in H$ be vectors such that 
$\omega_{\xi,\xi}|_{\cl N} = \omega_{\eta,\eta}|_{\cl N}$. Then there exists a partial isometry $V\in \cl N'$ such that 
$V\xi = \eta$ and $V^*V \xi = \xi$.  
\end{lemma}

\begin{theorem}\label{th_vncase}
Every von Neumann algebra possesses property (V).
\end{theorem}
\begin{proof}
The proof is split into several steps. 

\medskip

\noindent {\it Step 1. } 
Assume that $\cl A\subseteq \cl B(H)$ is in standard form, and denote by $\cl P$ the 
corresponding positive cone of vectors in $H$ (see \cite[Chapter IX]{tII} for details regarding 
the standard form of a von Neumann algebra). 
Suppose, in addition, that $x,y\in \cl P$. 
By \cite[Theorem IX.1.2]{tII}, 
\begin{equation}\label{eq_ars1} 
\|x - y\|^2 \leq \|(\omega_{x} - \omega_{y})|_{\cl A'}\|.
\end{equation}
Write 
$(\omega_x-\omega_y)_+$ (resp. $(\omega_x-\omega_y)_-$) for the 
positive (resp. negative) part of $(\omega_x-\omega_y)|_{\cl A'}$. 
Let $P$ be the projection onto the support of $(\omega_x-\omega_y)_-$; then $P\in \cl A'$. 
Taking into account that $x,y$ belong to $\cl P$ and using inequality (\ref{eq_ars1}), we have
\begin{eqnarray*}
2(x,y) 
& \geq & 
\omega_x(I) + \omega_y(I)  - ((\omega_x - \omega_y)(P^{\perp}) - (\omega_x - \omega_y)(P))\\
& = & 
2(\omega_x(P) + \omega_y(P^{\perp})),
\end{eqnarray*}
which implies (\ref{eq_vna}).

\medskip

\noindent {\it Step 2. } 
We still assume that $\cl A\subseteq \cl B(H)$ is in standard form, and let $x,y\in H$ be arbitrary. 
There exist vectors $|x|$ and $|y|$ in $\cl P$ such that 
$\omega_x|_{\cl A'} = \omega_{|x|}|_{\cl A'}$ and $\omega_y|_{\cl A'} = \omega_{|y|}|_{\cl A'}$
(see \cite[Theorem IX.1.2 (iv)]{tII}). 
By Lemma \ref{l_pi}, there exist partial isometries $U,V\in \cl A$ such that 
$U|x| = x$ and $V|y| = y$; note that, in addition, $U^*x = |x|$ and $V^*y = |y|$. 
Thus, 
\begin{eqnarray*}
\|\omega_{x,y}|_{\cl A}\| & = & 
\sup \{|(Ax,y)| : A\in {\rm Ball}(\cl A)\}\\
& = & 
\sup \{|(V^*AU|x|,|y|)| : A\in {\rm Ball}(\cl A)\} 
\leq \|\omega_{|x|,|y|}|_{\cl A}\|.
\end{eqnarray*}
By symmetry, $\|\omega_{x,y}|_{\cl A}\| = \|\omega_{|x|,|y|}|_{\cl A}\|$. 

Similarly, for every projection $Q\in \cl A'$, we have 
$$\|Qx\| = \|QU|x|\| = \|UQ|x|\| \leq \|Q|x|\|.$$
By symmetry, $\|Qx\| = \|Q|x|\|$ and $\|Qy\| = \|Q|y|\|$. Thus, (\ref{eq_vna}) follows from Step 1. 

\medskip

\noindent {\it Step 3. } 
Assume that $\cl M\subseteq \cl B(H)$ is a von Neumann algebra in standard form, 
$K$ is a Hilbert space, 
$\cl A = \cl M \otimes 1_K$, 
$\xi\in H$, $z\in K$ with $\|z\| = 1$, $x = \xi \otimes z$ and $y\in H\otimes K$. 
Write $Q$ for the projection onto $\bb{C}z$. 
Let $\eta\in H$ and $y_0$ be such that $(I\otimes Q)y_0 = 0$ and $y = \eta \otimes z + y_0$. 
By Step 2, there exists a projection $P \in \cl M'$ such that 
\begin{equation}\label{eq_forM}
\|P\xi\|^2 + \|P^{\perp}\eta\|^2 \leq \|\omega_{\xi,\eta}|_{\cl M}\|.
\end{equation}
Set $L = P\otimes Q + I \otimes Q^\perp$; clearly, $L$ is a projection in $\cl A'$.
Moreover, 
$\omega_{x,y}|_{\cl A} = \omega_{\xi,\eta}|_{\cl M}$,
$\|Lx\| = \|P\xi\|$ and $\|L^{\perp}y\| = \|P^{\perp}\eta\|$.
Thus, (\ref{eq_forM}) implies 
$\|Lx\|^2 + \|L^{\perp}y\|^2 \leq \|\omega_{x,y}|_{\cl A}\|$.

\medskip

\noindent {\it Step 4. } 
Assume that $\cl M\subseteq \cl B(H)$ is a von Neumann algebra in standard form, 
$K$ is a Hilbert space, 
$\cl A = \cl M \otimes 1_K$ and $x,y\in H\otimes K$. 
Since $\cl M$ is in standard form, there exists a vector $\xi\in H$ such that 
$$\omega_x(A\otimes I) = \omega_{\xi}(A), \ \ \ \ A\in \cl M.$$
Let $z\in K$ be any unit vector. We thus have 
$$\omega_x(A\otimes I) = \omega_{\xi\otimes z}(A\otimes I), \ \ \ \ A\in \cl M.$$
By Lemma \ref{l_pi}, there exists a partial isometry 
$V\in \cl A'$ such that 
$Vx = \xi\otimes z$ and $V^*V x = x$. 
By Lemma \ref{l_reduc} and Step 3, (\ref{eq_vna}) holds for the pair $(x,y)$. 

\medskip

\noindent {\it Step 5. } Let $\cl A$ be arbitrary, and 
let $\cl M\subseteq \cl B(H)$ be its standard form. 
By \cite[Theorem IV.5.5]{tII}, $\cl A$ is unitarily equivalent to the algebra 
$E(\cl M\otimes 1_K)E$, for some projection $E\in (\cl M\otimes 1_K)'$. 
It follows from Step 4 and Lemma \ref{l_induc} that (V) holds for $\cl A$. 
\end{proof}

We next turn our attention to CSL algebras.

\begin{theorem}\label{th_arv}
Every CSL algebra possesses property (V).
\end{theorem}
\begin{proof}
Let $\cl L$ be a CSL on $H$ and let $\cl A = \Alg \cl L$. 
Fix $x,y\in H$. 
We first prove the statement in the case $\cl L$ is finite.
Let $P\in\cl L$ be such that
\begin{equation}\label{eq_max}
\|P^{\perp} x\|^2 + \|Py\|^2 \leq \|L^{\perp} x\|^2 + \|Ly\|^2, \ \ \ \mbox{ for every } L\in \cl L.
\end{equation}
Set $K_1 = PH$, $K_2 = P^{\perp}H$, $x_1 = Px$, $x_2 = P^{\perp} x$, $y_1 = Py$, $y_2 = P^{\perp}y$, and
let $\cl L_1 = \{L|_{K_1} : L\in \cl L\}$ and $\cl L_2 = \{L|_{K_2} : L\in \cl L\}$ be
the restrictions of $\cl L$ to $K_1$ and $K_2$, respectively. (Note that, since $K_1$ and $K_2$ are invariant for each $L\in \cl L$,
these restrictions are well-defined.)

We claim that
\begin{equation}\label{eq_new}
\|L_1^{\perp} y_1\|^2 \leq \|L_1^{\perp} x_1\|^2, \ \ \ L_1\in \cl L_1.
\end{equation}
To show this, assume that 
$\|PL^{\perp} y\|^2 > \|PL^{\perp} x\|^2$ for some $L\in \cl L$.
Then
\begin{eqnarray*}
\|(PL)^{\perp} x\|^2 + \|PL y\|^2 & = & \|P^{\perp}L x\|^2 + \|P^{\perp} L^{\perp} x\|^2 + \|PL^{\perp} x\|^2 + \|PL y\|^2\\
& < & \|P^{\perp}L x\|^2 + \|P^{\perp} L^{\perp} x\|^2 + \|PL^{\perp} y\|^2 + \|PL y\|^2\\
& = & \|P^{\perp} x\|^2 + \|P y\|^2,
\end{eqnarray*}
which contradicts (\ref{eq_max}) and hence (\ref{eq_new}) is established. By \cite{hop},
there exists $T_1\in \ball(\Alg\cl L_1)$ such that $T_1x_1 = y_1$.

Similarly, we claim that
\begin{equation}\label{eq_new2}
\|L_2x_2\|^2 \leq \|L_2y_2\|^2, \ \ \ L_2\in \cl L_2.
\end{equation}
To show (\ref{eq_new2}) suppose, by way of contradiction, that 
$\|L P^{\perp} x\|^2  > \|LP^{\perp} y\|^2$ for some $L\in \cl L$. 
Then
\begin{eqnarray*}
& & \|(P\vee L)^{\perp} x\|^2 + \|(P\vee L) y\|^2\\ 
& = & \|P^{\perp}L^{\perp} x\|^2 + \|P^{\perp} L y\|^2 + \|PL^{\perp} y\|^2 + \|PL y\|^2\\
& < & \|P^{\perp}L^{\perp} x\|^2 + \|P^{\perp} L x\|^2 + \|PL^{\perp} y\|^2 + \|PL y\|^2
 =  \|P^{\perp} x\|^2 + \|P y\|^2,
\end{eqnarray*}
which contradicts (\ref{eq_max}) and hence (\ref{eq_new2}) is established. By \cite{hop},
there exists $T_2\in \ball((\Alg \cl L_2)^*)$ such that $T_2y_2 = x_2$.
Let $T = T_1\oplus T_2^*\in \cl B(K_1\oplus K_2)$. 
For every $L\in \cl L$, we have that
\begin{eqnarray*}
L^{\perp}TL 
& = & 
(L^{\perp}P\oplus L^{\perp}P^{\perp}) (T_1\oplus T_2^*) (LP\oplus LP^{\perp})\\ 
& = &
(L^{\perp}PT_1LP)\oplus (L^{\perp}P^{\perp}T_2^*LP^{\perp}) = 0,
\end{eqnarray*}
and so $T\in \ball(\cl A)$.
Also,
\begin{eqnarray*}
\|\omega_{x,y}|_{\cl A}\| 
& \geq & 
|\omega_{x,y}(T)| = |((PTP + P^{\perp}TP^{\perp})x,y)|\\ 
& = & 
|(T_1x_1,y_1) + (T_2^*x_2,y_2)| = |(T_1x_1,y_1) + (x_2,T_2y_2)|\\ 
& = & 
\|y_1\|^2 + \|x_2\|^2 
= \|P^{\perp} x\|^2 + \|Py\|^2 = d((x,y),\cl E_{\cl L})^2.
\end{eqnarray*}
Thus, the claim of the theorem is proved in the case $\cl L$ is finite.

Now assume $\cl L$ is an arbitrary CSL. It is 
straightforward that there exists a sequence $(\cl L_n)_{n=1}^{\infty}$
of finite CSL's such that $\cl L_n\subseteq \cl L_{n+1}$ for each $n\in \bb{N}$ and $\cup_{n=1}^{\infty} \cl L_n$
is strongly dense in $\cl L$. 
We claim that 
\begin{equation}\label{eq_ineqq}
\inf\mbox{}_{X\in \Conv \cl L} \left(((I-X)x,x) + (Xy,y)\right) \leq \liminf\mbox{}_{n\in \bb{N}} d\left((x,y),\cl E_{\cl L_n}\right)^2.
\end{equation}
To see this, 
suppose that $d\left((x,y),\cl E_{\cl L_{n_k}}\right)^2\rightarrow_{k\to\infty }\delta$, for some 
strictly increasing sequence $(n_k)_{k\in \bb{N}} \subseteq \bb{N}$. 
Let $\epsilon > 0$ and $k_0\in \bb{N}$ be such that 
\begin{equation}\label{eq_suse}
\|L^{\perp}_{n_k} x\|^2 + \|L_{n_k}y\|^2 \leq \delta+\epsilon, \ \ \ k\geq k_0.
\end{equation}
Since the unit ball of $\cl B(H)$ is compact in the weak operator topology,
we may assume, after passing to a subsequence if necessary, that $L_{n_k}\rightarrow_{k\to\infty} A$
weakly, for some $A\in\cl B(H)$. By assumption, 
$A\in\Conv\cl L$. Moreover,
$$\|L^{\perp}_{n_k}x\|^2 + \|L_{n_k}y\|^2 \rightarrow_{k\to\infty} ((I-A)x,x) + (Ay,y),$$ 
and now (\ref{eq_suse}) implies 
$$((I-A)x,x) + (Ay,y)\leq \delta+\epsilon.$$
Inequality (\ref{eq_ineqq}) follows by letting 
$\epsilon$ tend to zero.

Let $\cl A_n = \Alg \cl L_n$, $n\in \bb{N}$; then $\cap_{n=1}^{\infty}\cl A_n = \cl A$.
We have that 
\begin{equation}\label{eq_dec}
\limsup\mbox{}_{n\in \bb{N}} \|\omega_{x,y}|_{\cl A_{n}}\| \leq \|\omega_{x,y}|_{\cl A}\|.
\end{equation}
Indeed, suppose that 
$$\|\omega_{x,y}|_{\cl A_{n_k}}\| \rightarrow_{k\to\infty} \delta,$$ 
for some strictly increasing sequence $(n_k)_{k\in \bb{N}} \subseteq \bb{N}$. 
Let $T_{k}\in \ball(\cl A_{n_k})$ be such that 
$\|\omega_{x,y}|_{\cl A_{n_k}}\| = (T_k x,y)$.
Passing to a subsequence if necessary, we may assume that 
$T_k\rightarrow_{k\to\infty} T$ in the weak operator topology, for some $T\in \cl B(H)$.
Then $T\in \ball(\cl A)$ and so
$$\delta = \lim\mbox{}_{k\to\infty} (T_k x,y) = (Tx,y)\leq \|\omega_{x,y}|_{\cl A}\|.$$

By the first part of the proof, inequality (\ref{eq_el}) holds for each of the algebras $\cl A_n$.
Lemma \ref{l_conhulls} and inequalities (\ref{eq_ineqq}) and (\ref{eq_dec}) now imply that (\ref{eq_el}) holds for $\cl A$.
\end{proof}

\noindent {\bf Remark } 
Suppose that an operator algebra $\cl A\subseteq \cl B(H)$ has property (V)
and let $x,y\in H$, $\|x\| = \|y\| = 1$.
Then a necessary and sufficient condition for the existence of an operator $T\in \ball(\cl A)$
with $Tx = y$ is the validity of the inequalities 
\begin{equation}\label{eq_Lperp}
\|L^{\perp} y\|\leq \|L^{\perp}x\|, \ \ \  L\in \Lat \cl A.
\end{equation}
Indeed, it is straightforward to check that the inequalities 
(\ref{eq_Lperp}) are necessary. 
Conversely, assuming (\ref{eq_Lperp}), we have
$$\|L^{\perp} x\|^2 + \|Ly\|^2 \geq \|L^{\perp}y\|^2 + \|Ly\|^2 = \|y\|^2 = 1, \ \ L\in \Lat\cl A,$$ 
and since $\cl A$ is assumed to have property (V), by Theorem \ref{th_arv},
$\|\omega_{x,y}|_{\cl A}\| \geq 1$. Since $x$ and $y$ are unit vectors,
$\|\omega_{x,y}|_{\cl A}\| = 1$, and hence there exists $T\in \ball(\cl A)$ such that $(Tx,y) = 1$.
Thus, we have equality in the Cauchy-Schwarz inequality and hence $\lambda Tx = y$ for some (unimodular) scalar $\lambda$.

We note, however, that the observation in the previous paragraph cannot be used to
give a different proof of the necessary and sufficient
conditions for the solution of the equation $Tx = y$ given in \cite{hop} since
the result of \cite{hop} was used in the proof of Theorem \ref{th_arv}.

\section{Validity of (V') and violation of (V)}\label{s_vtonos}

In this section, we exhibit an example of a weakly closed unital operator 
algebra without property (V). We say that an operator algebra $\cl A\subseteq \cl B(H)$ 
satisfies \emph{property {\rm (V')}} if
there exists $c > 0$ such that 
\begin{equation}\label{eq_el2}
\inf\mbox{}_{L\in \Lat\cl A} \left(\|L^{\perp}x\|^2 + \|Ly\|^2\right) \leq c\|\omega_{x,y}|_{\cl A}\|,
\end{equation}
for all $x,y\in H$. 
Property (V') is clearly weaker that (V) and can be thought of as a quantitative version of the latter.

Let $N$ and $M$ be closed subspaces of a Hilbert space $H$. 
Following Halmos \cite{ha}, we say that $N$ and $M$ are in generic position if
$$N \cap M = N^{\perp}\cap M = M^{\perp}\cap N = M^{\perp}\cap N^{\perp} = \{0\}.$$
We say that the angle between $N$ and $M$ is positive if the algebraic sum $N + M$ is 
closed; otherwise, we say that the angle between $N$ and $M$ is zero.

It follows from \cite{ha} that if $N$ and $M$ are closed 
subspaces of a Hilbert space $H$ in generic position, then
there exists a Hilbert space $H_0$ such that,
up to unitary equivalence, 
$H = H_0\oplus H_0$,
$$N = \{(x, Bx): x \in H_0\} \ \ \mbox{ and } \ \  M = \{(x, -Bx): x \in H_0\}$$ 
where the operator $B \in \cl B(H_0)$ satisfies the conditions
\begin{enumerate}
 \item[(a)]
 $0 \leq B \leq I$;
 \item[(b)]
 $\ker B=\ker (I-B)=\{0\}$.
 \end{enumerate}
Moreover, the angle between $N$ and $M$ is positive if and only if the operator $B$ is invertible.

Let $P$ (resp. $Q$) be the orthogonal projection onto $N$ (resp. $M$).
Writing 
$$\Gamma = \left[ \begin{matrix}
(I+B^2)^{-1} &  0 \\
0 &  (I+B^2)^{-1} \\
\end{matrix} \right],$$
we have
$$P = 
\Gamma \left[ \begin{matrix}
I &  B \\
 B &  B^2 \\
\end{matrix} \right], \ \ 
P^{\perp}= 
\Gamma \left[ \begin{matrix}
B^2 &  -B \\
 -B &  I \\
\end{matrix} \right],
$$

$$Q = 
\Gamma \left[ \begin{matrix}
I &  -B \\
 -B &  B^2 \\
\end{matrix} \right], \ \mbox{ and } \ 
Q^{\perp} = 
\Gamma \left[ \begin{matrix}
B^2 &  B \\
 B &  I \\
\end{matrix} \right].
$$

In the rest of the section, we denote by 
$\cl A$ the algebra of operators on $H$ that leave $N$ and $M$ invariant; 
note that $\Lat\cl A = \{0,P,Q,I\}$. 

Let 
$$\cl B = \left\{T=
\left[ \begin{matrix}
C &  0 \\
0 &  D \\
\end{matrix} \right] :  C, D \in \cl B(H_0)\right\}.$$
We will need the following result (see \cite{p}). 

\begin{proposition}\label{p_sim}
Suppose that the angle between $N$ and $M$ is positive. 
Set $a=\sqrt{2}/2$ and 
$$S=
\left[ \begin{matrix}
B^{-\frac{1}{2}} &  0 \\
 0 &  B^{\frac{1}{2}} \\
\end{matrix} \right]
\left[ \begin{matrix}
aI &  -aI \\
 aI &  aI \\
\end{matrix} \right].$$
Then the operator $S$ is invertible and $\cl A=S\cl B S^{-1}$. 
\end{proposition}

Fix $x_1, x_2, y_1, y_2 \in H_0$, and let $x = (x_1, x_2)$, $y = (y_1, y_2)$. 
Assuming that the angle between $N$ and $M$ is positive, let
$$m_1 = a^2\|B^{\frac{1}{2}}x_1+B^{-\frac{1}{2}}x_2\|^2 + a^2 \|B^{\frac{1}{2}}x_1-B^{-\frac{1}{2}}x_2\|^2,$$
$$m_2 = a^2 \|B^{\frac{1}{2}}x_1+B^{-\frac{1}{2}}x_2\|^2 + a^2 \|B^{-\frac{1}{2}}y_1 - B^{\frac{1}{2}}y_2\|^2,$$
$$m_3 = a^2 \|B^{-\frac{1}{2}}y_1+B^{\frac{1}{2}}y_2\| ^2  + a^2 \|B^{\frac{1}{2}}x_1 - B^{-\frac{1}{2}}x_2\|^2$$
and 
$$m_4 = a^2 \|B^{-\frac{1}{2}}y_1+B^{\frac{1}{2}}y_2\|^2  +  a^2 \|B^{-\frac{1}{2}}y_1 - B^{\frac{1}{2}}y_2\|^2.$$

\begin{lemma}\label{propothem}
We have that 
$\min\{m_1,m_2,m_3,m_4\}\leq \|\omega_{S^{-1}x, S^*y}|_{\cl B}\|.$
\end{lemma}
\begin{proof}
Note that 
 $$S^{-1}x=
 \left[ \begin{matrix}
 a(B^{\frac{1}{2}}x_1+B^{-\frac{1}{2}}x_2)\\
 a(-B^{\frac{1}{2}}x_1+B^{-\frac{1}{2}}x_2) \\
\end{matrix}\right]$$
and
$$S^{*}y=
 \left[ \begin{matrix}
 a(B^{-\frac{1}{2}}y_1+B^{\frac{1}{2}}y_2)\\
 a(-B^{-\frac{1}{2}}y_1+B^{\frac{1}{2}}y_2) \\
\end{matrix}\right].$$
The assertion follows from the fact that 
$\cl B$ satisfies property (V) (see Theorem \ref{th_vncase}). 
\end{proof}

Set
$$a_1=\|x_1\|^2+\|x_2\|^2, \ 
a_2=\|P^{\perp}x\|^2 + \|Py\|^2,$$
$$a_3=\|Q^{\perp}x\|^2 + \|Qy\|^2 \ \mbox{ and } \
a_4=\|y_1\|^2+\|y_2\|^2.$$

\begin{lemma}\label{lemmathem}
There exists $c \geq 0$ such that
$$a_1\leq c m_1, a_2\leq c m_3, a_3\leq c m_2 \mbox{ and }   a_4\leq c m_4.$$
In particular, there exists $c \geq  0$ with 
$$\min\{a_1, a_2, a_3, a_4\}\leq c \min\{m_1, m_2, m_3, m_4\}.$$
\end{lemma}
\begin{proof}
We have 
\begin{eqnarray*}
a_1
& = & 
\|x_1\|^2+\|x_2\|^2\leq \|B^{-1}\|^2\|Bx_1\|^2+\|x_2\|^2\\
& \leq & 
\|B^{-1}\|^2(\|Bx_1\|^2+\|x_2\|^2)\\
& = & 
\frac{1}{2}\|B^{-1}\|^2(\|Bx_1-x_2\|^2+\|Bx_1+x_2\|^2)\\
& \leq & 
\frac{1}{2}\|B^{-1}\|^2(\|B^{\frac{1}{2}}(B^{\frac{1}{2}}x_1-B^{-\frac{1}{2}}x_2)\|^2+\|B^{\frac{1}{2}}(B^{\frac{1}{2}}x_1+B^{-\frac{1}{2}}x_2)\|^2)\\
& \leq & 
\frac{1}{2}\|B^{-1}\|^2\|B^{\frac{1}{2}}\|^2(\|B^{\frac{1}{2}}x_1-B^{-\frac{1}{2}}x_2\|^2+\|B^{\frac{1}{2}}x_1+B^{-\frac{1}{2}}x_2)\|^2)\\
& \leq & \frac{1}{2}\|B^{-1}\|^2\|B^{\frac{1}{2}}\|^2a^{-2}m_1,
\end{eqnarray*}

\begin{eqnarray*}
a_2
& = & 
\|P^{\perp}x\|^2+ \|Py\|^2\\
& = &  
\left\| 
\Gamma \left[ \begin{matrix}
B^2 &  -B \\
 -B &  I \\
\end{matrix} \right]\left[ \begin{matrix}
x_1  \\
 x_2 \\
\end{matrix} \right]\right\|^2
+ 
\left\|
\Gamma 
\left[ \begin{matrix}
I &  B \\
 B &  B^2 \\
\end{matrix} \right]
 \left[ \begin{matrix}
y_1  \\
 y_2 \\
\end{matrix}\right]\right\|^2\\
& \leq & 
\|\Gamma\|^2\left(\left\| 
 \left[ \begin{matrix}
B^2 &  -B \\
 -B &  I \\
\end{matrix} \right]\left[ \begin{matrix}
x_1  \\
 x_2 \\
\end{matrix} \right]\right\|^2
+ 
\left\|
\left[ \begin{matrix}
I &  B \\
 B &  B^2 \\
\end{matrix} \right]
 \left[ \begin{matrix}
y_1  \\
 y_2 \\
\end{matrix}\right]\right\|^2\right)\\
& \leq & 
\|\Gamma\|^2\left(\left\| 
 \left[ \begin{matrix}
B^2x_1   -Bx_2 \\
 -Bx_1 +  x_2 \\
\end{matrix} \right]\right\|^2
+ 
\left\|
\left[ \begin{matrix}
y_1 +  By_2 \\
 By_1  + B^2y_2 \\
\end{matrix} \right]
 \right\|^2\right)\\
& \leq & 
\|\Gamma\|^2(\|
B^2x_1   -Bx_2 \|^2+
\|-Bx_1 +  x_2 \|^2\\
& + & 
\|y_1 +  By_2\|^2 +\|By_1  + B^2y_2 \|^2)\\
& \leq & 
\|\Gamma\|^2(\|B\|^2\|
Bx_1   -x_2 \|^2+
 \|Bx_1 -  x_2 \|^2\\
& + &  
\|y_1 +  By_2\|^2 +\|B\|^2\| y_1  + By_2 \|^2)\\
& \leq & 
2\|\Gamma\|^2\left(\|
Bx_1   -x_2 \|^2+
\|
 y_1  + By_2 \|^2
\right)\\
& \leq & 
2\|\Gamma\|^2\left(\|B^{\frac{1}{2}}(
B^{\frac{1}{2}}x_1   -B^{-\frac{1}{2}}x_2) \|^2+
\|B^{\frac{1}{2}}(
 B^{-\frac{1}{2}}y_1  + B^{\frac{1}{2}}y_2) \|^2
\right)\\
& \leq & 
2\|\Gamma\|^2\|B^{\frac{1}{2}}\|^2\left(\|
B^{\frac{1}{2}}x_1   -B^{-\frac{1}{2}}x_2 \|^2+
\|
 B^{-\frac{1}{2}}y_1  + B^{\frac{1}{2}}y_2 \|^2
\right)\\
& \leq & 2\|\Gamma\|^2 a^{-2}m_3,
\end{eqnarray*}

\begin{eqnarray*}
a_3
& = & 
\|Q^{\perp}x\|^2+ \|Qy\|^2\\
& =  & 
\left\| 
\Gamma \left[ \begin{matrix}
B^2 &  B \\
 B &  I \\
\end{matrix} \right]\left[ \begin{matrix}
x_1  \\
 x_2 \\
\end{matrix} \right]\right\|^2
+ 
\left\|
\Gamma 
\left[ \begin{matrix}
I &  -B \\
 -B &  B^2 \\
\end{matrix} \right]
 \left[ \begin{matrix}
y_1  \\
 y_2 \\
\end{matrix}\right]\right\|^2\\
& \leq & 
\|\Gamma\|^2\left(\left\| 
 \left[ \begin{matrix}
B^2 &  B \\
 B &  I \\
\end{matrix} \right]\left[ \begin{matrix}
x_1  \\
 x_2 \\
\end{matrix} \right]\right\|^2
+ 
\left\|
\left[ \begin{matrix}
I &  -B \\
 -B &  B^2 \\
\end{matrix} \right]
 \left[ \begin{matrix}
y_1  \\
 y_2 \\
\end{matrix}\right]\right\|^2\right)\\
& = & 
\|\Gamma\|^2\left(\left\| 
 \left[ \begin{matrix}
B^2x_1   +Bx_2 \\
 Bx_1 +  x_2 \\
\end{matrix} \right]\right\|^2
+ 
\left\|
\left[ \begin{matrix}
y_1 -  By_2 \\
 -By_1  + B^2y_2 \\
\end{matrix} \right]
 \right\|^2\right)\\
& \leq & 
\|\Gamma\|^2(\|
B^2x_1   + Bx_2 \|^2+
 \|Bx_1 +  x_2 \|^2\\
& + & 
\|y_1 -  By_2\|^2 +\|By_1  - B^2y_2 \|^2)\\
& \leq & 
\|\Gamma\|^2(\|B\|^2\|
Bx_1   + x_2 \|^2+
 \|Bx_1 +  x_2 \|^2\\
& + & 
\|y_1 -  By_2\|^2 +\|B\|^2\| y_1  - By_2 \|^2)\\
& \leq & 
2\|\Gamma\|^2\left(\|Bx_1   +x_2 \|^2+\|y_1  - By_2 \|^2\right)\\
& \leq & 
2\|\Gamma\|^2\left(\|B^{\frac{1}{2}}(
B^{\frac{1}{2}}x_1   +B^{-\frac{1}{2}}x_2) \|^2+
\|B^{\frac{1}{2}}(
 B^{-\frac{1}{2}}y_1  - B^{\frac{1}{2}}y_2) \|^2\right)\\
& \leq & 
2\|\Gamma\|^2\|B^{\frac{1}{2}}\|^2\left(\|B^{\frac{1}{2}}x_1  +B^{-\frac{1}{2}}x_2\|^2
+\|B^{-\frac{1}{2}}y_1  - B^{\frac{1}{2}}y_2 \|^2
\right)\\
& \leq & 
2\|\Gamma\|^2 a^{-2} m_2
\end{eqnarray*}

and 

\begin{eqnarray*}
a_4
& = & 
\|y_1\|^2+\|y_2\|^2\\
& \leq & 
\|y_1\|^2+\|B^{-1}\|^2\|By_2\|^2\leq \|B^{-1}\|^2(\|y_1\|^2+\|By_2\|^2)\\
& = & 
\frac{1}{2}\|B^{-1}\|^2(\|y_1+By_2\|^2+\|y_1-By_2\|^2)\\
& = & 
\frac{1}{2}\|B^{-1}\|^2(\|B^{\frac{1}{2}}(B^{-\frac{1}{2}}y_1 + B^{\frac{1}{2}}y_2)\|^2+\|B^{\frac{1}{2}}(B^{-\frac{1}{2}}y_1-B^{\frac{1}{2}}y_2)\|^2)\\
& \leq & 
\frac{1}{2}\|B^{-1}\|^2\|B^{\frac{1}{2}}\|^2(\|B^{-\frac{1}{2}}y_1+B^{\frac{1}{2}}y_2\|^2+\|B^{-\frac{1}{2}}y_1-B^{\frac{1}{2}}y_2\|^2)\\
& \leq & \frac{1}{2}\|B^{-1}\|^2a^{-2}m_4.
\end{eqnarray*}
\end{proof}

\begin{theorem}\label{th_vtonos}
The algebra $\cl A$  satisfies property (V') if and only if the angle between $N$ and $M$ is positive.
\end{theorem}
\begin{proof}
Suppose that the angle between $N$ and $M$ is positive.
Let $x,y\in H$. 
By Lemmas \ref{propothem} and \ref{lemmathem}, there exists $c \geq 0$
such that 
$$\inf\mbox{}_{L\in \Lat\cl A} (\|L^{\perp}x\|^2 + \|Ly\|^2) 
\leq c \|\omega_{S^{-1}x, S^*y}|_{\cl B}\|.$$
On the other hand, 
letting $t = \|S\|\|S^{-1}\|$ we have by Proposition \ref{p_sim} that 
\begin{eqnarray*}
\|\omega_{S^{-1}x, S^*y}|_{\cl B}\|
& = & 
\sup_{B \in \cl B, \|B\|\leq 1} |\omega_{x, y}(SBS^{-1})|\\
& \leq & 
\sup_{A \in \cl A, \|A\|\leq t} |\omega_{x, y}(A)| = t \|\omega_{x, y}|_{\cl A}\|.
\end{eqnarray*}
It follows that $\cl A$ satisfies (V'). 

Assume that the angle between $N$ and $M$ is zero. 
Let $x_1, y_2\in H_0$ be unit vectors and
$x = (x_1, 0)$ and $y = (0, y_2)$ with respect to the decomposition $H = H_0\oplus H_0$.
We have
\begin{eqnarray*}
\|P^{\perp}x\|^2+ \|Py\|^2 
& = & 
\left\| 
\Gamma \left[ \begin{matrix}
B^2 &  -B \\
 -B &  I \\
\end{matrix} \right]\left[ \begin{matrix}
x_1  \\
 0 \\
\end{matrix} \right]\right\|^2
+ 
\left\|
\Gamma 
\left[ \begin{matrix}
I &  B \\
 B &  B^2 \\
\end{matrix} \right]
 \left[ \begin{matrix}
0  \\
 y_2 \\
\end{matrix}\right]\right\|^2\\
\end{eqnarray*}
\begin{eqnarray*}
& = & 
\left\| 
\Gamma  \left[ \begin{matrix}
B^2x_1  \\
 -Bx_1 \\
\end{matrix} \right]\right\|^2
+ 
\left\|
\Gamma  \left[ \begin{matrix}
By_2 \\
 B^2y_2 \\
\end{matrix}\right]\right\|^2.
\end{eqnarray*}

Since 
$$\left\|\left[ \begin{matrix}
(I+B^2) &  0 \\
 0 &  (I+B^2) \\\end{matrix} \right]\right \|^{-1}\geq 1/2$$
and, for every invertible operator $X $ and every vector $\xi$ we have 
$\|X\xi \|\geq \|X^{-1}\|^{-1}\|\xi\|$, 
we obtain
$$\|P^{\perp}x\|^2+ \|Py\|^2\geq 
\frac{1}{4}\left(\left\| \left[ \begin{matrix}
B^2x_1  \\
 -Bx_1 \\
\end{matrix} \right]\right\|^2+ \left\|
 \left[ \begin{matrix}
By_2 \\
 B^2y_2 \end{matrix}\right]\right\|^2\right)\geq \frac{1}{4}\|By_2\|^2.$$

A similar calculation gives
$$\|Q^{\perp}x\|^2+ \|Qy\|^2\geq 
\frac{1}{4}\left(\left\| \left[ \begin{matrix}
B^2x_1  \\
 Bx_1 \\
\end{matrix} \right]\right\|^2+ \left\|
 \left[ \begin{matrix}
-By_2 \\
 B^2y_2 \\
\end{matrix}\right]\right\|^2\right)\geq \frac{1}{4}\|By_2\|^2.$$
Hence
\begin{equation}\label{eq_B2}
\inf_{L\in \Lat\cl A} \left(\|L^{\perp} x\|^2 + \|Ly\|^2\right) \geq  \frac{1}{4}\|By_2\|^2.
\end{equation}

We calculate the right hand side of inequality (\ref{eq_el2}).
Let $T \in \ball(\cl A)$. We may write
$$T=
\left[ \begin{matrix}
C &  D \\
 BDB &  R \\
\end{matrix} \right]$$
for some $C, D, R  \in \cl B(H_0)$ such that $BC = RB$ (see {\it e.g.} \cite{akl}).
We have
$\omega_{x, y}(T)=(BDBx_1, y_2)$
and, since $D$ is a contraction, 
$$|\omega_{x,y}(T)|\leq \|Bx_1\|\|By_2\|.$$

Assume now that there exists a constant $c$ such that  
\begin{equation*}
\inf_{L\in \Lat\cl A} \left(\|L^{\perp} x\|^2 + \|Ly\|^2\right) \leq c \|\omega_{x,y}|_{\cl A}\|.
\end{equation*}
Then, by (\ref{eq_B2}),
\begin{equation}\label{eq_bb}
 \frac{1}{4}\|By_2\|^2  \leq c  \|Bx_1\|\|By_2\|.
\end{equation}

By assumption, $B$ is injective and not invertible. Denoting by $E(\cdot)$ its spectral measure, 
we can hence find a decreasing sequence $(\lambda_n)_{n=0}^{\infty}\subseteq (0,1]$ such that 
$\lim_{n\to\infty} \lambda_n = 0$ and the projections 
$E_n = E([\lambda_n,\lambda_{n-1}))$ are non-zero for all $n\geq 1$. 
For each $n\geq 1$, let $e_n$ be a unit vector with $E_n e_n = e_n$. 
Taking
$x_1 = e_{n+1}$ and $y_2 = e_1$ in (\ref{eq_bb}), we obtain
\begin{equation*}
\frac{1}{4}\lambda_1^2\leq  c \lambda_0 \lambda_n.
\end{equation*}
Letting $n \rightarrow \infty$ we obtain a contradiction with the fact that
$\lim_{n \rightarrow \infty} \lambda_n = 0$.
\end{proof}

\begin{remark}
{\rm Theorem \ref{th_vtonos} shows that if $H_0$ is finite dimensional then 
$\cl A$ automatically has property (V').}
\end{remark}

\end{document}